\documentclass[11pt]{article}
\usepackage{smile}

\usepackage{fullpage}
\usepackage{framed}
\usepackage{mdframed}
\usepackage{enumerate}
\usepackage[inline]{enumitem}
\usepackage[T1]{fontenc}
\usepackage{moresize}
\usepackage{bm}
\usepackage{dsfont}
\usepackage{amsmath}
\usepackage{amssymb}
\usepackage{amsthm}
\usepackage{amsfonts}
\usepackage{stmaryrd}
\usepackage{array}
\usepackage{mathrsfs}
\usepackage{mathtools} 
\usepackage{extarrows}
\usepackage{stackrel}
\usepackage[nodisplayskipstretch]{setspace}
\usepackage{color}
\usepackage[usenames,dvipsnames]{xcolor}
\usepackage{cancel}
\usepackage{soul}
\usepackage{xfrac}
\usepackage{siunitx}
\usepackage{graphicx}
\usepackage{float}
\usepackage{rotating}
\usepackage{subcaption}
\usepackage{overpic}
\usepackage[all]{xy}
\usepackage{tikz}
\usetikzlibrary{arrows,matrix,positioning,calc,automata,patterns}
\usepackage{booktabs}
\usepackage{dcolumn}
\usepackage{multirow}
\usepackage{diagbox}
\usepackage{lscape}
\usepackage{rotating}
\usepackage{verbatim}
\usepackage{listings}
\usepackage{fancyvrb}
\usepackage{hyperref}
\usepackage[round]{natbib}
\usepackage{sectsty}

\usepackage[linesnumbered,ruled]{algorithm2e}
\usepackage{algpseudocode}
\SetKwInput{KwInput}{Input}             
\SetKwInput{KwInitialization}{Initialization}      
\SetKwInput{KWReturn}{Return}

\hypersetup{
    bookmarks=true,         
    unicode=false,          
    pdftoolbar=true,        
    pdfmenubar=true,        
    pdffitwindow=false,     
    pdfstartview={FitH},    
    pdftitle={My title},    
    pdfauthor={Author},     
    pdfsubject={Subject},   
    pdfcreator={Creator},   
    pdfproducer={Producer}, 
    pdfkeywords={key1, key2}, 
    pdfnewwindow=true,      
    colorlinks=true,        
    linkcolor=blue,         
    citecolor=blue,         
    filecolor=blue,         
    urlcolor=cyan           
}

\def\tr{\mathop{\text{tr}}\kern.2ex}

\def\E{{\mathbb E}}

\def\ind{{\mathds 1}}

\newcolumntype{L}[1]{>{\raggedright\let\newline\\\arraybackslash\hspace{0pt}}m{#1}}
\newcolumntype{C}[1]{>{  \centering\let\newline\\\arraybackslash\hspace{0pt}}m{#1}}
\newcolumntype{R}[1]{>{ \raggedleft\let\newline\\\arraybackslash\hspace{0pt}}m{#1}}
\newcolumntype{d}[1]{D{.}{.}{#1}}
\newcolumntype{H}{>{\setbox0=\hbox\bgroup}c<{\egroup}@{}}
\newcolumntype{Z}{>{\setbox0=\hbox\bgroup}c<{\egroup}@{\hspace*{-\tabcolsep}}}
\renewcommand{\d}{\textrm{\sf d}}

\numberwithin{equation}{section}
\newtheorem{theorem}{Theorem}[section]
\newtheorem{lemma}{Lemma}[section]
\newtheorem{proposition}{Proposition}[section]
\newtheorem{assumption}{Assumption}[section]

\newtheorem{innercustomgeneric}{\customgenericname}

\providecommand{\customgenericname}{}
\newcommand{\newcustomtheorem}[2]{%
  \newenvironment{#1}[1]
  {%
   \renewcommand\customgenericname{#2}%
   \renewcommand\theinnercustomgeneric{##1}%
   \innercustomgeneric
  }
  {\endinnercustomgeneric}
}
\newcustomtheorem{customtheorem}{Theorem}
\newcustomtheorem{customassumption}{Assumption}
\newcustomtheorem{customlemma}{Lemma}
\newcustomtheorem{customexample}{Example}
\theoremstyle{definition}

\allowdisplaybreaks

\begin{document}

\setlength{\abovedisplayskip}{5pt}
\setlength{\belowdisplayskip}{5pt}
\setlength{\abovedisplayshortskip}{5pt}
\setlength{\belowdisplayshortskip}{5pt}
\hypersetup{colorlinks,breaklinks,urlcolor=blue,linkcolor=blue}

\title{\LARGE Smoothed NPMLEs in nonparametric Poisson mixtures and beyond}

\author{
Keunwoo Lim\thanks{Department of Applied Mathematics, University of Washington, Seattle, WA 98195, USA. E-mail: \tt{kwlim@uw.edu}} ~~and~ Fang Han\thanks{Department of Statistics, University of Washington, Seattle, WA 98195, USA; e-mail: {\tt fanghan@uw.edu}}}

\date{\today}

\maketitle


\begin{abstract} 
We discuss nonparametric mixing distribution estimation under the Gaussian-smoothed optimal transport (GOT) distance. It is shown that a recently formulated conjecture---that the Poisson nonparametric maximum likelihood estimator can achieve root-$n$ rate of convergence under the GOT distance---holds up to some logarithmic terms. We also establish the same conclusion for other minimum-distance estimators, and discuss mixture models beyond the Poisson.
\end{abstract}


\section{Introduction and main results} \label{sec:intro}

One of the main issues in mixture models pertains to estimating the unknown mixing distribution. This question is well understood for parametric models, for which the problem reduces to point estimation in finite dimensions. Substantial progress has been made in understanding nonparametric mixture models \citep{chen2017consistency}, particularly recently in the cases of Gaussian \citep{wu2020optimal,doss2020optimal} and Poisson \citep{Miao24}. However, many open problems remain in this area; one concerning nonparametric Poisson mixtures is discussed here.

Consider the Poisson mixture model with {\it mixture density function} $h_Q^{\rm Pois}(x)$ given by
\begin{align}\label{eq:poisson-model}
h_Q^{\rm Pois}(x) = \int_{0}^{\theta_{\ast}} {\rm Pois}(x;\theta) \d Q(\theta), ~~x=0,1,2,\ldots.
\end{align}
Here ${\rm Pois}(x;\theta):=e^{-\theta}\theta^x/x!$ is the Poisson probability mass function (PMF), $\theta_\ast>0$ is a {\it known} upper bound of $Q(\cdot)$, and $Q(\cdot)$ represents the {\it unknown mixing distribution} of support in $[0,\theta_\ast]$.  The object of interest is to estimate $Q(\cdot)$ based only on $n$  random integers $X_1,\ldots,X_n$ independently sampled from $h_Q^{\rm Pois}(\cdot)$.

In \cite{Miao24}, the authors studied  the following nonparametric maximum likelihood estimator (NPMLE) of $Q$,
\[
\hat Q_{\rm MLE}:= \argmax_{\tilde Q \text{ over }[0,\theta_\ast]}\sum_{i=1}^n\log h_{\tilde Q}^{\rm Pois}(X_i),
\]
where the maximum is taken over all probability distributions of support over $[0,\theta_\ast]$. They proved, denoting $W_1(\cdot,\cdot)$ the Wasserstein-1 distance,
\[
\sup_{Q \text{ over }[0,\theta_\ast]}\E W_1(Q,\hat Q_{\rm MLE}) \leq C\cdot \frac{\log\log n}{\log n},
\]
with the constant $C=C(\theta_\ast)$ only depending on $\theta_\ast$. The authors also constructed a worst-case example, showing that this bound cannot be further improved in the minimax sense.

More recently, there has been growing interest in quantifying the estimation accuracy of $\hat Q_{\rm MLE}$ using a smoothed optimal transport distance as an alternative to $W_1$. Specifically, let $\cN(0,\sigma^2)$ denote the normal distribution with a fixed variance $\sigma^2>0$ and let ``$\ast$'' represent the convolution operation. Define the Gaussian-smoothed optimal transport (GOT) distance as 
\[
W_1^\sigma(Q,Q'):=W_1\Big(Q\ast \cN(0,\sigma^2),Q'\ast \cN(0,\sigma^2)\Big).
\]
This distance was introduced in recent works \citep{goldfeld2020convergence,goldfeld2020gaussian,Goldfeld24,goldfeld2024limit}. It was shown in \citet[Theorem 7 and Remark 10]{Han23} that 
\begin{align}\label{eq:hms23}
\sup_{Q \text{ over }[0,\theta_\ast]}\E W_1^\sigma\Big(Q,\hat Q_{\rm MLE}\Big) \leq C(\theta_\ast, \sigma^{2},c)n^{-1/10+c},
\end{align}
where $c>0$ can be chosen arbitrarily small. Clearly, the subpolynomial rate under the $W_1$ distance is accelerated exponentially under $W_1^{\sigma}$.

Unfortunately, the $O(n^{-1/10})$-rate of convergence is not satisfactory. In the same paper \citep{Han23}, it was conjectured that the bound \eqref{eq:hms23} is loose and, up to some logarithmic terms, should be improvable to the parametric root-$n$ rate. Indeed, the proof of \citet[Theorem 7]{Han23}, based on a Jackson-type approximation bound \citep{jackson1921general} of Gaussian-smoothed Lipschitz functions, is aimed for generic nonparametric mixture distributions. There should be room to leverage improvements specific to the Poisson case.

In a follow-up paper, \cite{Teh23} showed that improvements to \eqref{eq:hms23} can be made by exploiting the special structure of the Poisson. Their Corollary 2.1 improved on \cite{Han23} in a substantial way, proving that, for arbitrarily small $c>0$,
\begin{align}\label{eq:tp23}
\sup_{Q \text{ over }[0,\theta_\ast]}\E W_1^\sigma\Big(Q,\hat Q_{\rm MLE}\Big) \leq C(\theta_\ast, \sigma^{2},c)n^{-1/4+c}.
\end{align}

The main result of this paper states that \cite{Han23}'s conjecture holds up to some logarithmic terms. 

\begin{theorem}\label{thm:main} Under the Poisson mixture model \eqref{eq:poisson-model}, it holds true that
\begin{align*}
    \sup_{Q \text{ over }[0, \theta_{\ast}])}\mathbb{E}W_{1}^{\sigma}(Q, \hat{Q}_{\text{MLE}})=
    \begin{cases}
        O(n^{-1/2+c/\log \log n}\,\textnormal{Polylog}(n)),\,\,\,\,\,&\text{if}\,\,\sigma^{2} <4\theta_{\ast},\\
        O(n^{-1/2}\,\textnormal{Polylog}(n)), &\text{if}\,\,\sigma^{2} \geq 4\theta_{\ast},
    \end{cases}
\end{align*}
with $c = 8 \log (4\theta_{\ast}/\sigma^{2})$. 
\end{theorem}
Similar to \cite{Teh23}, the above result is viable by making use of the structure of the Poisson distribution, for which we introduce a new approximation lemma tailored specifically (Lemma \ref{lemma: polynomial}). This new approximation bound turns out to be sufficiently powerful to further sharpen the analysis in \eqref{eq:tp23}, and indicates further that the comparison theorem in \citet[Theorem 1]{Teh23} could be loose. The corresponding bound in Theorem \ref{thm:main} is also interestingly similar to the remarkable lower bound established in \citet[Theorem 21]{polyanskiy2021sharp}; both are nearly parametric but  ${\rm Polylog}(n)$ away.

Subsequent sections further extend the main result to a class of minimum-distance estimators that include the NPMLE as a special example. Additionally, in a less satisfactory way, these results are extended to the discrete exponential family, including the Poisson as a special case. The presentation of these results is deferred to the following sections.

\paragraph{Notation.} \label{notation}
Denote $\mathbb{R}$, $\mathbb{R}^{>0}$, $\mathbb{R}^{\geq 0}$, $\mathbb{Z}^{\geq 0}$ as the sets of all real numbers, of all positive real numbers, of all nonnegative real numbers, and of all nonnegative integers, respectively. Let $\mathcal{P}(A)$ be the set of all probability measures on the domain $A\subset\mathbb{R}$. Let $\ind(\cdot)$ represent the indicator function.

\section{Models and methods}\label{section:models}

Starting from this section, we consider a more general setting than the Poisson and assume the following nonparametric mixture model:
\begin{align}\label{eq:model-general}
    h_{Q}(x) = \int_{0}^{\theta_{\ast}} f(x; \theta) \d Q(\theta), ~~~ x\in \mathbb{Z}^{\geq 0}.
\end{align}
Here $f(\cdot; \theta)$ is a {\it known} parametric PMF, not necessarily the Poisson, and $Q(\cdot)$ is the {\it unknown} mixing distribution. Let $X_1,\ldots,X_n$ be independently drawn from $h_Q$. 

In this paper we further require $f(x;\theta)$ to belong to the {\it discrete exponential family} \citep{Zhang95, Han23}, for which $f(x;\theta)$ admits a special decomposition form,
\begin{align}\label{eq:model-general2}
    f(x; \theta) = g(\theta)w(x)\theta^{x}, ~~~ x\in \mathbb{Z}^{\geq 0} ~~{\rm and}~~\theta\in[0, \theta_{\ast}].
\end{align}
It is additionally assumed that $g(\theta)$ is analytic in the neighborhood of $0$, and $0 \leq \theta_{\ast}<\theta_{c}$ where $\theta_{c}$ is the radius of convergence of $\sum_{x\geq 0} w(x)\theta^{x}$. It can be shown that Poisson distribution satisfies all the above conditions and thus is a special case.

Next, we introduce the minimum-distance estimators of $Q$. To this end, let's consider a generalized distance $d$ of any $p, q\in \mathcal{P}(\mathbb{Z}^{\geq 0})$ as a mapping
\[
d(p \Vert q) \colon \mathcal{P}(\mathbb{Z}^{\geq 0})\times\mathcal{P}(\mathbb{Z}^{\geq 0})\to \mathbb{R}^{\geq 0}, 
\]
such that $d(p \Vert q) = 0$ if and only if $p = q$. Wolfowitz's minimum-distance estimators, introduced in \cite{Wolfowitz53, Wolfowitz57},  aim to minimize the distance between the observed distribution and estimated distribution, measured by $d(\cdot,\cdot)$, as
\begin{align}\label{eq:minimizer}
\hat Q_d \in \argmin_{\tilde Q\in \cP([0,\theta_\ast])}d(h^{\rm obs} \Vert h_{\tilde Q}).
\end{align}
Here the observed empirical distribution $h^{\rm obs}(\cdot)$ is given by
\[
    h^{\text{obs}}(x) = \frac1n \sum_{i = 1}^{n}\ind(X_{i} = x),~~~\text{ for all }x\in\mathbb{Z}^{\geq 0}.
\]
In the special case when we choose $d$ to be the Kullback–Leibler divergence
\[
d(p\Vert q)= {\rm KL}(p\Vert q):=\sum_{x\geq 0} p(x) \log \frac{p(x)}{q(x)},
\]
the corresponding minimum-distance estimator $\hat Q_{\rm KL}$ is well known to be equivalent to the NPMLE
\[
\hat{Q}_{\text{MLE}} = \argmax_{\tilde Q\in\mathcal{P}([0, \theta_{\ast}])}\frac{1}{n}\sum_{i = 1}^{n} \log h_{\tilde Q}(X_{i}).
\]
Thusly, the NPMLE belongs to the general family of minimum-distance estimators.

Lastly, consider the computation of $\hat Q_d$. To this end, an additional structure for $d(\cdot,\cdot)$ is needed. 

\begin{assumption}\label{assumption: structure} There exist maps $w \colon \mathcal{P}(\mathbb{Z}^{\geq 0})\to \mathbb{R} $ and $\phi\colon \mathbb{R}^{\geq 0}\times \mathbb{R}^{\geq 0}\to \mathbb{R}$ such that the generalized distance $d(p \Vert q)$ can be decomposed as 
\begin{align}\label{eq:distance form}
    d(p \Vert q) = w(p) +\sum_{x \geq 0}\phi(p(x), q(x)).
\end{align}
In addition, assume $\phi$ to be continuously differentiable, $\phi(0,y_{2}) = 0$ for any $y_2\geq 0$, and $\phi(y_{1}, y_{2})$ is a strictly convex function of $y_{2}$ for any fixed $y_{1}>0$.
\end{assumption}

Let the set of values of the observations be $\{i_{1},\dots,i_{q}\}$ and the number of observations with value $i_{x}$ as $a_{x} = \sum_{s = 1}^{n}\ind(X_{s} = i_{x})$. In addition, let $\alpha_{x} = a_{x}/n$. The empirical distribution $h^{\text{obs}}$ can then be written as $h^{\text{obs}}(i_{x}) = \alpha_{x}$ and zero elsewhere. Define 
\[
\mu_{x}(G) = \int_{0}^{\theta_{\ast}} f(i_{x}; \theta)\, \d G(\theta) ~~{\rm and}~~ \mu(G) = (\mu_{1}(G),\dots,\mu_{q}(G))^\top. 
\]
Next, introduce 
\[
\Psi(u) = \sum_{x = 1}^{q}\phi(\alpha_{x},u_{x}) ~~{\rm and}~~ \Phi(G) = \Psi \circ \mu(G)=\sum_{x=1}^q\phi(\alpha_x,\mu_x(G)).
\] 
Lastly, let $\delta_{\lambda}$ be a point mass on $\lambda$ and denote the directional derivative $\Phi'(G, P)$ of the probability measure $G\in \mathcal{P}([0, \theta_{\ast}])\backslash \{\delta_{0}\}$ with regard to the direction of the probability measure $P$ as
\begin{align*}
    \Phi'(G, P) = \lim_{\epsilon \to 0^{+}}\frac{\Phi((1-\epsilon)G+\epsilon P) - \Phi(G)}{\epsilon}.
\end{align*}

With these notation, we are then ready to introduce the Vertex Direction Method (VDM, \cite{Simar76}) and Intra Simplex Direction Method (ISDM, \cite{lesperance1992algorithm}) for computing the minimum-distance estimators, outlined in Algorithms \ref{algorithm:vdm} and \ref{algorithm:isdm}.

\begin{algorithm}[!ht]
\caption{Vertex Direction Method (VDM)}  \label{algorithm:vdm}
\LinesNotNumbered
  \KwInput{data points $X_{1},\dots,X_{n}$ and $\lambda_{0}\in (0, \theta_{\ast})$}
  \KwInitialization{initial distribution $G = \delta_{\lambda_{0}}$}
  \While{$\min_{\lambda\in[0, \theta_{\ast}]}\Phi'(G, \delta_{\lambda})< 0$}
  {choose $\lambda\in\argmin_{\lambda\in[0, \theta_{\ast}]}\Phi'(G, \delta_{\lambda})$\\
  choose $\epsilon\in\argmin_{\epsilon\in[0,1]}\Phi((1-\epsilon)G+\epsilon\delta_{\lambda})$\\
  update $G\leftarrow (1-\epsilon)G+\epsilon\delta_{\lambda}$
  }
 \KWReturn{$G$} 
\end{algorithm}

\begin{algorithm}[!ht]
\caption{Intra Simplex Direction Method (ISDM)}\label{algorithm:isdm}
\LinesNotNumbered
  \KwInput{data points $X_{1},\dots,X_{n}$ and $\lambda_{0}\in (0, \theta_{\ast})$}
  \KwInitialization{initial distribution $G = \delta_{\lambda_{0}}$}
  \While{$\min_{\lambda\in[0, \theta_{\ast}]}\Phi'(G, \delta_{\lambda})< 0$}
  {choose local minima $\lambda_{1},\dots,\lambda_{s}$ of $\Phi'(G, \delta_{\lambda})$\\
  choose $(\epsilon_{0},\dots,\epsilon_{s})\in\argmin_{\epsilon_{x}\geq 0, \sum\epsilon_{x} = 1}\Phi(\epsilon_{0}G+\sum_{x=1}^{s}\epsilon_{x}\delta_{\lambda_{x}})$\\
  update $G\leftarrow \epsilon_{0}G+\sum_{x=1}^{s}\epsilon_{x}\delta_{\lambda_{x}}$
  }
 \KWReturn{$G$}
\end{algorithm}

\section{Theory}

To establish our theory for the estimator $\hat Q_d$, we first regularize the generalized distance. To this end, we introduce the following two assumptions, which are analogous to the assumptions in \citet[Section 2]{Jana22}. 

\begin{assumption}\label{assumption: KLbound}
There exist some universal constants $\gamma_1,\gamma_2>0$ such that the generalized distance $d$  satisfies, for any $p,q\in \mathcal{P}(\mathbb{Z}^{\geq 0})$,
\begin{align*}
 \gamma_{1}H^{2}(p,q) \leq d(p \Vert q) \leq \gamma_{2} \text{\rm KL}(p \Vert q),
\end{align*}
where $H^{2}(p,q) = \frac{1}{2}\sum_{x\geq 0}(\,\sqrt{p(x)} - \sqrt{q(x)}\,)^{2}$ is the squared Hellinger distance.
\end{assumption}

In the special case of Poisson mixture, Assumption \ref{assumption: KLbound} can be weakened to the following.

\begin{assumption}\label{assumption: chibound}
There exist some universal constants $\gamma_3,\gamma_4>0$ such that the generalized distance $d$  satisfies, for any $p,q\in \mathcal{P}(\mathbb{Z}^{\geq 0})$,
\begin{align*}
    \gamma_{3}H^{2}(p,q) \leq d(p \Vert q) \leq \gamma_{4} \chi^{2}(p \Vert q),
\end{align*}
where $\chi^{2}(p\Vert q) = \sum_{x\geq 0}(p(x) - q(x))^{2}/q(x)$ is the chi-square divergence.
\end{assumption}

\begin{proposition}\label{proposition: distance}
The squared Hellinger distance, Le Cam distance ($\frac{1}{2}\sum_{x \geq 0}\frac{(p(x) - q(x))^{2}}{p(x)+q(x)}$), Jensen-Shannon divergence (${\rm KL}(p\Vert \frac{p+q}{2})+{\rm KL}(q \Vert \frac{p+q}{2})$), KL-divergence, and chi-square distance all satisfy Assumptions \ref{assumption: structure} and \ref{assumption: chibound}. In addition, all satisfy Assumption \ref{assumption: KLbound} except for the chi-square distance.
\end{proposition}

The first theory concerns the convergence of Algorithms \ref{algorithm:vdm} and \ref{algorithm:isdm}.

\begin{theorem}\label{theorem: algorithm}
Assume Model \eqref{eq:model-general} with $f(x;\theta)$ taking the form \eqref{eq:model-general2} and being continuous with regard to $\theta$. Further assume Assumption \ref{assumption: structure} holds. Then, the minimum-distance estimator $\hat{Q}_{d}$ exists, and VDM and ISDM yield a solution path $\Phi(G)\to\Phi(\hat{Q}_{d})$ as the number of iterations increases.    
\end{theorem}

The following is the general theory on $\hat Q_d$ assuming \eqref{eq:model-general} and \eqref{eq:model-general2}. Note that it does not require Assumption \ref{assumption: structure} and applies to any minimizer (not necessarily unique) of \eqref{eq:minimizer}.

\begin{theorem}\label{theorem: exponential family}
Assume \eqref{eq:model-general} and \eqref{eq:model-general2} and Assumption \ref{assumption: KLbound}. It then holds true that
\begin{align*}
    \sup_{Q \in \mathcal{P}([0, \theta_{\ast}])}\mathbb{E}W_{1}^{\sigma}(Q, \hat{Q}_{d}) \leq C n^{-c},
\end{align*}
where $C=C(\theta_\ast,\sigma^2,\gamma_1,\gamma_2)$ and $c=c(\theta_\ast,\sigma^2)$ are two positive constants.
\end{theorem}

Specializing to the Poisson case, the following result establishes (a generalized version of) the main theorem, Theorem \ref{thm:main}.

\begin{theorem}\label{theorem: poisson}
In the case of the Poisson mixture model \eqref{eq:poisson-model}, assuming further Assumption \ref{assumption: chibound}, we have    
\begin{align*}
    \sup_{Q \in \mathcal{P}([0, \theta_{\ast}])}\mathbb{E}W_{1}^{\sigma}(Q, \hat{Q}_{d})\leq 
    \begin{cases}
        C(\theta_{\ast}, \sigma^{2}, \gamma_{3}, \gamma_{4})\cdot n^{-0.5+c/\log \log n}\,\textnormal{Polylog}(n),\,\,\,\,\,&\text{if}\,\,\sigma^{2} <4\theta_{\ast},\\
        C(\theta_{\ast}, \sigma^{2}, \gamma_{3}, \gamma_{4})\cdot n^{-0.5}\,\textnormal{Polylog}(n)), &\text{if}\,\,\sigma^{2} \geq 4\theta_{\ast},
    \end{cases}
\end{align*}
holds with $c = 8 \log (4\theta_{\ast}\sigma^{-2})$. 
\end{theorem}

In particular, taking $d(\cdot,\cdot)$  to be the KL divergence, Theorem \ref{theorem: poisson} yields Theorem \ref{thm:main}.

\section{Proofs}\label{section:proofs_6}

Let's start the proof section with some additional notation. Let the convolution between the functions $f(x)$ and $g(x)$ be given as
\begin{align*}
    [f\ast g](x) = \int_{\mathbb{R}} f(x-y)g(y)\d y,
\end{align*}
and let $\varphi_{\sigma}(\theta)$ be the probability density function of a normal distribution $\mathcal{N}(0, \sigma^{2})$. Denote $\ell(\cdot)$ as a $1$-Lipschitz function with $\ell(0) = 0$ and denote $\ell_{\sigma}(\theta) := [\ell \ast \varphi_{\sigma}](\theta)$. Then, a dual representation of $W_{1}(Q_{1}, Q_{2})$ for probability measures $Q_{1}$ and $Q_{2}$ over $[0, \theta_{\ast}]$ is
\begin{align*}
    W_{1}(Q_{1}, Q_{2}) = \sup_{\ell(\cdot)}\int_{0}^{\theta_{\ast}} \ell(\theta)\,(\d Q_{1} - \d Q_{2}),
\end{align*}
and the dual representation of $W_{1}^{\sigma}(Q_{1}, Q_{2})$ is
\begin{align*}
    W_{1}^{\sigma}(Q_{1}, Q_{2}) = \sup_{\ell(\cdot)}\int_{0}^{\theta_{\ast}} (\ell_{\sigma}(\theta)-\ell_{\sigma}(0))\,(\d Q_{1} - \d Q_{2}).
\end{align*}

\subsection{Proofs of main results}
\begin{proof}[Proof of Theorem \ref{theorem: algorithm}:]
We follow the proof of \citet[Section 3.2]{Simar76} and \citet[Theorem 1]{Jana22} on the existence of the estimator, and \citet[Section 2]{Bohning82} and \citet[Theorem 3.1]{Miao24} for the convergence of the algorithms. \\

\par\noindent\textbf{Step 1.} In this step, we prove the existence of the minimum-distance estimator. Recalling the function $\mu(G)$, its image $\mathcal{U}$ is convex and compact since it is bounded and closed by Prokhorov's Theorem \citep[Section 5]{billingsley2013convergence}. In addition, since $\Psi(u)$ is strictly convex, there exists a minimizer $\hat{u} = (\hat{u}_{1},\dots,\hat{u}_{q})$ of $\Psi$ on $\mathcal{U}$ with some measure $\hat{G}$ that satisfies $\hat{u} = \mu(\hat{G})$. It leads to the existence of the minimum-distance estimator $\hat{G}$. \\

\par\noindent\textbf{Step 2.} Then, we consider the VDM and analyze its convergence. First, we establish the relation between the directional derivative $\Phi'(G, \delta_{\lambda})$ and the finite difference $\Phi((1-\epsilon)G + \epsilon \delta_{\lambda}) - \Phi(G)$. We obtain the decrease of the function $\Phi$ in the algorithm in the next step based on this relation.

 Recall the directional derivative for probability measure $G\in \mathcal{P}([0, \theta_{\ast}])\backslash \{\delta_{0}\}$ with probability measure $P$ as
\begin{align*}
    \Phi'(G, P) = \lim_{\epsilon \to 0^{+}}\frac{\Phi((1-\epsilon)G+\epsilon P) - \Phi(G)}{\epsilon}= \sum_{x = 1}^{q}\frac{\partial \phi}{\partial y_{2}}(\alpha_{x}, \mu_{x}(G))(\mu_{x}(P) - \mu_{x}(G)), 
\end{align*}
which has specific relation with $\phi$ in this case. Our main statement in this step is that there exists some $\epsilon_1=\epsilon_{1}(\alpha, \kappa)$ such that
\begin{align*}
    \Phi((1-\epsilon)G + &\epsilon \delta_{\lambda}) - \Phi(G) \leq -\epsilon \kappa / 2   
\end{align*}
holds for every $0<\epsilon<\epsilon_{1}$, $G\in\mathcal{P}([0, \theta_{\ast}])$, and $\lambda\in[0,\theta_{\ast}]$ with conditions $\Phi(G)\leq \alpha$ and $\Phi'(G, \delta_{\lambda})\leq -\kappa<0$. Since $\Phi = \Psi \circ \mu$, there exists $\eta \in [0, 1]$ such that
\begin{align*}
    \Phi((1-\epsilon)G &+ \epsilon \delta_{\lambda}) - \Phi(G) - \epsilon \Phi'(G, \delta_{\lambda})\\ &= \epsilon(\nabla\Psi((1-\eta\epsilon)\mu(G) + \eta\epsilon\mu(\delta_{\lambda}))-\nabla \Psi(\mu(G)))^{\top}(\mu(\delta_{\lambda}) - \mu(G))
\end{align*}
by the mean value theorem and direct computation. Denote the compact set 
\[
\mathcal{U}_{\alpha} = \{ u\in\mathcal{U}\, \vert \, \Psi(u)\leq \alpha \} 
\]
and consider $\mu(G)\in \mathcal{U}_{\alpha}\subset \mathcal{U}_{\alpha+1}$. We note that $\nabla \Psi$ is continuous and $\mathcal{U}_{\alpha+1}$ is compact, so $\nabla \Psi$ is uniformly continuous on $\mathcal{U}_{\alpha+1}$. In addition, since $\Vert\mu(\delta_{\lambda}) - \mu(G)\Vert_{2}\leq 2$ holds, there exists $\epsilon_{1}(\alpha, \zeta)$ such that $(1-\eta \epsilon) \mu(G) + \eta \epsilon \mu(\delta_{\lambda})\in \mathcal{U}_{\alpha+1}$ and 
\begin{align*}
    \vert \nabla\Psi((1-\eta\epsilon)\mu(G) + \eta\epsilon\mu(\delta_{\lambda}))-\nabla \Psi(\mu(G))\vert< \zeta
\end{align*}
for all $\mu(G) \in \mathcal{U}_{\alpha}$ and $0<\epsilon<\epsilon_{1}$. Finally, denoting $\mathcal{P}_{\alpha} = \{G\in\mathcal{P}([0, \theta_{\ast}])\vert\Phi(G) \leq \alpha\}$ analogous to $\mathcal{U}_{\alpha}$, for every $G\in \mathcal{P}_{\alpha}$ that satisfies $\Phi'(G, \delta_{\lambda})\leq -\kappa$ and for every $0<\epsilon<\epsilon_{1}(\alpha, \kappa /4)$, we then have  
\begin{align*}
    \Phi((1-\epsilon)G + &\epsilon \delta_{\lambda}) - \Phi(G) \\&\leq \epsilon \Phi'(G, \delta_{\lambda}) + \epsilon\vert\nabla\Psi((1-\eta\epsilon)\mu(G) + \eta\epsilon\mu(\delta_{\lambda}))-\nabla \Psi(\mu(G))\vert\vert\mu(\delta_{\lambda}) - \mu(G)\vert\\ &\leq -\epsilon\kappa +\epsilon\kappa/2 = -\epsilon \kappa / 2   
\end{align*}
holds. It leads to the bound of the difference between the finite difference and directional derivative, which is essential in analyzing the algorithm. \\

\par\noindent\textbf{Step 3.} Finally, we analyze the convergence of the function $\Phi$ during the iteration. In particular, letting the value of $G$ on the $N$-th iteration of the algorithm be $G_{N}$ and recalling the minimizer $\hat{G}$, we prove that $\Phi(G_{N})\to \Phi(\hat{G})$ during the iteration and $\Phi(G_{N}) = \Phi(\hat{G})$ holds when the algorithm finishes at the $N$-th stop.  

First, we prove that $\Phi(G_{N})\to \Phi(\hat{G})$ during the iteration. Since $\Phi(G_{N})$ is an decreasing sequence with lower bound $\Phi(\hat{G})$, there exists a limit of $\Phi(G_{N})$ which is $\lim_{N\to\infty}\Phi(G_{N}) = \Phi^{-}$. We assume that the limit $\Phi^{-}\neq \Phi(\hat{G})$ and prove that it leads to a contradiction. Considering the first step of the algorithm and letting $\lambda_{\text{min}}\in \argmin_{\lambda \in[0, \theta_{\ast}]}\Phi'(G_{N}, \delta_{\lambda})$, we obtain the upper bound of $\Phi'(G_{N}, \delta_{\lambda_{\text{min}}})$ with $\Phi'(G_{N}, \hat{G})$ such that
\begin{align*}
    \Phi'(G_{N}, \delta_{\lambda_{\text{min}}}) &= \min_{\lambda\in[0, \theta_{\ast}]} \sum_{x = 1}^{q}\frac{\partial \phi}{\partial y_{2}}(\alpha_{x}, \mu_{x}(G))(f(i_{x}; \lambda) - \mu_{x}(G))\\
    &\leq\int_{0}^{\theta_{\ast}} \big(\sum_{x = 1}^{q}\frac{\partial \phi}{\partial y_{2}}(\alpha_{x}, \mu_{x}(G))(f(i_{x}; \lambda) - \mu_{x}(G))\big)\d \hat{G}(\lambda)\\
    &=\sum_{x = 1}^{q}\frac{\partial \phi}{\partial y_{2}}(\alpha_{x}, \mu_{x}(G))(\mu_{x}(\hat{G}) - \mu_{x}(G))= \Phi'(G_{N}, \hat{G}).
\end{align*}
By the convexity of $\Phi$,
\begin{align*}
    \Phi'(G_{N}, \hat{G}) \leq \Phi(\hat{G}) - \Phi(G_{N})\leq\Phi(\hat{G}) - \Phi^{-}<0
\end{align*}
holds and we derive that 
\[
\Phi'(G_{N}, \delta_{\lambda_{\text{min}}})\leq\Phi(\hat{G}) - \Phi^{-}<0.
\]
In addition,, since $\Phi(G_{N})$ is a decreasing with respect to the iteration $N$, there exists $N_{1}$ such that $\Phi(G_{N})\leq \alpha$ holds for every $N>N_{1}$ and some $\alpha$. Therefore, by the main statement in Step 2, there exists $\epsilon_{1}$ such that 
\begin{align*}
    \Phi(G_{N+1}) - \Phi(G_{N}) \leq \Phi((1-\epsilon)G_{N}+\epsilon \delta_{\lambda_{\text{min}}}) - \Phi(G_{N})\leq -\epsilon(\Phi^{-}-\Phi(\hat{G}))
\end{align*}
for every $N>N_{1}$ and $0<\epsilon<\epsilon_{1}$. It contradicts that $\Phi(G_{N})$ has a lower bound and it leads to $\Phi^{-}=\Phi(\hat{G})$ and $\Phi(G_{N})\to \Phi(\hat{G})$.

Then, we consider the case when the algorithm stops at iteration $N$ and prove that $\Phi(G_{N}) = \Phi(\hat{G})$. Since $\hat{G}$ is the minimizer of $\Phi$, $\Phi(G_{N}) \geq \Phi(\hat{G})$ holds and it suffices to prove that $\Phi(G_{N}) \leq \Phi(\hat{G})$. By the stopping condition of the algorithm, $\Phi'(G_{N}, \delta_{\lambda_{\text{min}}})\geq 0$ and 
\begin{align*}
    0\leq \Phi'(G_{N}, \delta_{\lambda_{\text{min}}}) \leq \Phi(\hat{G}) - \Phi(G_{N})
\end{align*}
holds by the above inequality, so $\Phi(G_{N}) \leq \Phi(\hat{G})$ and it completes the proof. Regarding the ISDM, we could follow the same proof since the local minima contain global minima.
\end{proof}

\begin{proof}[Proof of Theorem \ref{theorem: exponential family}:]

We follow the proofs of \citet[Theorem 7]{Han23} and \citet[Lemma 17]{Han23}, assuming that the generalized distance $d$ satisfies the Assumption \ref{assumption: KLbound} with constants $\gamma_{1}$ and $\gamma_{2}$. We also inherit the notation from these papers.

The proof consists of two steps. In the first step, we modify the approximation part on the original proof and make this approximation valid for the general minimum-distance estimators. Then, the next step is to derive the desired result with the previous result that the convergence rate of minimum-distance estimator for KL-divergence is $O(n^{-c})$ for some $c>0$.\\

\par\noindent\textbf{Step 1.}
In this step, we derive the approximation of minimum-distance estimators based on the original approximation. Based on the assumption and information-theoretical inequalities, one can prove that
\begin{align}\label{eq:exponential}
    \big\vert\sum_{x=0}^{2k} b_{x}[h^{\text{obs}}(x)-h_{\hat{Q}_{d}}(x)]\big\vert+\big\vert\sum_{x=0}^{2k} b_{x}[h^{\text{obs}}(x)-h_{Q}(x)]\big\vert\leq C k^{\frac{1}{2}} \max_{0\leq x \leq 2k}\vert b_{x}\vert\sqrt{\text{KL}(h^{\text{obs}}\Vert h_{Q})} 
\end{align}
for $b_{x}\in \mathbb{R}$ and sufficient constant $C = C(\gamma_{1},\gamma_{2})$ and relate it with the original proof. 

We start with bounding the first term of \eqref{eq:exponential}. By the triangle inequality, we could separate the coefficients $b_{x}$ and the partial sum of total variance as 
\begin{align*}
    \big\vert\sum_{x=0}^{2k} b_{x}[h^{\text{obs}}(x)-h_{\hat{Q}_{d}}(x)]\big\vert\leq \sum_{x=0}^{2k}\big\vert b_{x}\big \vert \big \vert h^{\text{obs}}(x)-h_{\hat{Q}_{d}}(x)\big\vert\leq \max_{0 \leq x \leq 2k}\vert b_{x}\vert\sum_{x=0}^{2k} \big \vert h^{\text{obs}}(x)-h_{\hat{Q}_{d}}(x)\big\vert.
\end{align*}
Then, by the factorization of the partial sum of total variance and Cauchy-Schwarz inequality, we have the Hellinger distance bound as
\begin{align*}
    \sum_{x = 0}^{2k}\big\vert h^{\text{obs}}(x) - h_{\hat{Q}_{d}}(x)\big \vert
    &\leq \big(\,\sum_{x = 0}^{2k} \big(\,\sqrt{h^{\text{obs}}(x)} + \sqrt{h_{\hat{Q}_{d}}(x)}\,\big)^{2}\,\big)^{\frac{1}{2}} \big(\,\sum_{x = 0}^{2k}\big(\,\sqrt{h^{\text{obs}}(x)} - \sqrt{h_{\hat{Q}_{d}}(x)}\,\big)^{2}\,\big)^{\frac{1}{2}}\\
    &\leq \sqrt{16k}\,\big(\,\frac{1}{2}\sum_{x = 0}^{2k}\big(\,\sqrt{h^{\text{obs}}(x)} - \sqrt{h_{\hat{Q}_{d}}(x)}\,\big)^{2}\,\big)^{\frac{1}{2}}\leq \sqrt{16k}\,H(h^{\text{obs}}, h_{\hat{Q}_{d}}). 
\end{align*}
Based on Assumption \ref{assumption: KLbound}, we have the relation between the Hellinger distance, generalized distance $d$, and KL-divergence as
\begin{align*}
    H(h^{\text{obs}}, h_{\hat{Q}_{d}})\leq \sqrt{\gamma_{1}^{-1}d(h^{\text{obs}}\Vert h_{\hat{Q}_{d}})}\leq \sqrt{\gamma_{1}^{-1}d(h^{\text{obs}}\Vert h_{Q})} \leq \sqrt{\gamma_{1}^{-1}\gamma_{2}\text{KL}(h^{\text{obs}}\Vert h_{Q})}
\end{align*}
where the second inequality is from the definition of minimum-distance estimator. Letting $C(\gamma_{1}, \gamma_{2}) = (16 \gamma_{1}^{-1}\gamma_{2})^{\frac{1}{2}}$, we obtain
\begin{align*}
    \big\vert\sum_{x=0}^{2k} b_{x}[h^{\text{obs}}(x)-h_{\hat{Q}_{d}}(x)]\big\vert&\leq \max_{0 \leq x \leq 2k}\vert b_{x}\vert\sum_{x=0}^{2k} \big \vert h^{\text{obs}}(x)-h_{\hat{Q}_{d}}(x)\big\vert\\
    &\leq \max_{0 \leq x \leq 2k}\vert b_{x}\vert\sqrt{16k}\,H(h^{\text{obs}}, h_{\hat{Q}_{d}})\leq C(\gamma_{1}, \gamma_{2})k^{\frac{1}{2}} \max_{0\leq x \leq 2k}\vert b_{x}\vert\sqrt{\text{KL}(h^{\text{obs}}\Vert h_{Q})}, 
\end{align*}
which is the bound for the first term of \eqref{eq:exponential}.

We could derive the similar result for the second term of \eqref{eq:exponential}. By the triangle inequality, we could obtain the bound with the partial sum of total variance as
\begin{align*}
    \big\vert\sum_{x=0}^{2k} b_{x}[h^{\text{obs}}(x)-h_{Q}(x)]\big\vert\leq \sum_{x=0}^{2k}\big\vert b_{x}\big \vert \big \vert h^{\text{obs}}(x)-h_{Q}(x)\big\vert\leq \max_{0 \leq x \leq 2k}\vert b_{x}\vert\sum_{x=0}^{2k} \big \vert h^{\text{obs}}(x)-h_{Q}(x)\big\vert. 
\end{align*}
In addition, by the similar derivation on factorization and Cauchy-Schwarz inequality, we have the bound with the Hellinger distance as
\begin{align*}
    \sum_{x = 0}^{2k}\big\vert h^{\text{obs}}(x) - h_{Q}(x)\big \vert \leq \sqrt{16k}\,\big(\,\frac{1}{2}\sum_{x = 0}^{2k}\big(\,\sqrt{h^{\text{obs}}(x)} - \sqrt{h_{Q}(x)}\,\big)^{2}\,\big)^{\frac{1}{2}}\leq \sqrt{16k}\,H(h^{\text{obs}}, h_{Q}). 
\end{align*}
Finally, applying the relation between Hellinger distance and KL-divergence
\begin{align*}
    H(h^{\text{obs}}, h_{Q}) \leq \sqrt{\frac{\text{KL}(h^{\text{obs}}\Vert h_{Q})}{2}}
\end{align*}
leads to the bound of the second term of \eqref{eq:exponential}
\begin{align*}
    \big\vert\sum_{x=0}^{2k} b_{x}[h^{\text{obs}}(x)-h_{Q}(x)]\big\vert&\leq \max_{0 \leq x \leq 2k}\vert b_{x}\vert\sum_{x=0}^{2k} \big \vert h^{\text{obs}}(x)-h_{Q}(x)\big\vert\\ &\leq \max_{0 \leq x \leq 2k}\vert b_{x}\vert\sqrt{16k}\,H(h^{\text{obs}}, h_{Q})\leq  \max_{0 \leq x \leq 2k}\vert b_{x}\vert\sqrt{16k}\,\sqrt{\text{KL}(h^{\text{obs}}\Vert h_{Q})}.
\end{align*}
Therefore, we could prove the argument by combining the above results. \\

\par\noindent\textbf{Step 2.} Next, we prove the polynomial convergence rate of Gaussian-smoothed Wasserstein distance of minimum-distance estimators $\mathbb{E} W_{1}^{\sigma}(Q, \hat{Q}_{d})$ based on the Step 3 of the proof of \citet[Theorem 7]{Han23}. 

Since the original case of minimum-distance estimator for KL-divergence have the polynomial convergence rate of $O(n^{-c})$, we prove that the convergence rate of our case of general minimum-distance estimator is $O(n^{-c}\log n)$, which also leads the polynomial convergence rate. We note that the parametric density function is given by $f(x; \theta) = g(\theta) w(x) \theta^{x}$ for discrete exponential family models. In addition, we analyze the asymptotics of $k$ in the proof with  the rate of $n$. Regarding the case (i), $k$ satisfies $C^{2k} = n^{\alpha}$ which implies that $k = O(\log n)$. For the case (ii), $k$ satisfies $(C_{1}k)^{2C_{2}k} = n^{\alpha}$, and since
\begin{align*}
    \log (C_{1}\log n)^{2 C_{2} \log n}= 2 C_{2} \log n \log (C_{1}\log n) \geq \alpha \log n = \log n^{\alpha}
\end{align*}
for all sufficiently large $n$, so $k = O(\log n)$ also holds.  

Based on the statement of the original proof, we have the upper bound of $W_{1}^{\sigma}(Q, \hat{Q}_{d})$ as
\begin{align*}
    W_{1}^{\sigma}(Q, \hat{Q}_{d})&\leq C\,\big(\,(2\theta_{\ast}e^{\frac{1}{2}}k^{\frac{1}{2}})^{-k}+\sum_{x\geq k+1}w(x)\theta_{\ast}^{x}+ke^{2k} \max\{1, 1/\theta_{\ast}\}^{2k}\frac{\max_{1\leq x \leq 2k}1/w(x)}{\sqrt{n^{1-\epsilon}\delta^{1+\epsilon}}}\,\big)\\
    &\leq Ck\,\big(\,(2\theta_{\ast}e^{\frac{1}{2}}k^{\frac{1}{2}})^{-k}+\sum_{x\geq k+1}w(x)\theta_{\ast}^{x}+e^{2k} \max\{1, 1/\theta_{\ast}\}^{2k}\frac{\max_{1\leq x \leq 2k}1/w(x)}{\sqrt{n^{1-\epsilon}\delta^{1+\epsilon}}}\,\big)
\end{align*}
with probability at least $1-\delta$ and constant $C = C(\theta_{\ast}, \sigma^{2}, \gamma_{1}, \gamma_{2})$. Then, following the same approach, the maximum ratio between the convergence rate of generalized distance case and the original case is $O(k)$, so the convergence rate of the generalized distance case is given by $O(n^{-c}\log n)$ for some $c>0$. It leads to the conclusion that we have $O(n^{-c})$ convergence rate with some $c>0$ for generalized distance $d$. In addition, we didn't used the specific condition of generalized distance $d$ except for Assumption \ref{assumption: KLbound}, so this rate holds for all generalized distance $d$ and is not related to the specific choice of $d$. 
\end{proof}

\begin{proof}[Proof of Theorem \ref{theorem: poisson}:]
The proof is based on the approach of the proof of \citet[Theorem 7]{Han23} for proving the minimum-distance estimator for KL-divergence in discrete exponential family models. 

We approximate the convoluted $1$-Lipschitz function $\ell_{\sigma}$ with $f(x; \theta)$ for the approximation of Gaussian-smoothed Wasserstein distance $W_{1}^ {\sigma}(Q, \hat{Q}_{d})$ and bound each terms by comparing $h_{Q}$, $h_{\hat{Q}_{d}}$, and $h^{\text{obs}}$. We have two major improvements regarding the convergence rate of $n$ and will present them in the proof. 

\par\noindent\textbf{Step 1.} First step is on the approximation of $W_{1}^{\sigma}(Q, \hat{Q}_{d})$ based on the approximation of $\ell_{\sigma}$ with $f(x; \theta)$, and the main idea is to approximate $\ell_{\sigma}$ with the linear combination of $\{e^{-\theta}\theta^{x}\}$ using Lemma \ref{lemma: polynomial}. For the approximation of the $1$-lipschitz function $\ell_{\sigma}$, we recall the polynomial 
\[
P_{k-1}(\theta) = \sum_{x = 0}^{k-1}c_{x}\theta^{x} 
\]
from Lemma \ref{lemma: polynomial} and denote $b_{x} =x!c_{x}$. To enable the appropriate approximation, let the degree of the polynomial $k-1$ satisfies $(k-1)^{k-1}<n^{2}\leq k^{k}$ while $n\geq \max(30, (e\sigma^{2})^{e\sigma^{2}})$. Then, by taking the logarithm to the inequality, we have $k \log k \geq 2 \log n$ and
\begin{align*}
    2 \log k \geq \log k + \log \log k \geq \log \log n + \log 2 \geq \log \log n,
\end{align*}
which leads $\log k \geq \log \log n / 2$. In addition, $k^{k}\geq n \geq 30$ implies $k \geq 3$ which leads $k-1> \sqrt{k}$ and $k-1> k/2$, so $k \log k \leq 8\log n$ holds. Combining the two inequalities, we have $k \leq 16 \log n / \log \log n$. Note that when $n\geq e^{e}$, 
\begin{align*}
    \big(\,\frac{\log n}{\log \log n}\,\big)^{\frac{\log n}{\log \log n}}\leq (\log n)^{\frac{\log n}{\log \log n}} = n
\end{align*}
so $k \geq \log n / \log \log n$ and $k = O(\log n / \log \log n)$ is a tight bound. 

We apply the polynomial approximation to $W_{1}^{\sigma}(Q, \hat{Q}_{d})$ based on the dual representation of $W_{1}^{\sigma}$ distance and obtain

\begin{align}
    W_{1}^{\sigma}(Q, &\hat{Q}_{d}) = \sup_{l}\int(\ell_{\sigma}(\theta)-\ell_{\sigma}(0))[\d Q(\theta) - \d \hat{Q}_{d}(\theta)]\nonumber\\
    &= \sup_{l}\int(\ell_{\sigma}(\theta)-\ell_{\sigma}(0)-\sum_{x = 0}^{k-1}b_{x}f(x; \theta))[\d Q(\theta) - \d \hat{Q}_{d}(\theta)]+\sum_{x = 0}^{k-1}b_{x}[h_{Q}(x)-h_{\hat{Q}_{d}}(x)].\label{eq: wasserstein}      
\end{align}
The first term of \eqref{eq: wasserstein} could be interpreted as an error of polynomial approximation from Lemma \ref{lemma: polynomial} and is bounded as
\begin{align*}
    \big\vert\,\int(\ell_{\sigma}(\theta)-\ell_{\sigma}(0)-\sum_{x = 0}^{k-1}b_{x}f(x; \theta))[\d Q(\theta) - \d \hat{Q}_{d}(\theta)]\,\big\vert\leq& \frac{C(\theta_{\ast}, \sigma^{2})(k+1)^{\frac{1}{2}} (2\theta_{\ast})^{ k} (1+ (k/e\sigma^{2})^{\frac{k}{2}})}{k!}\\
    \leq& \frac{C(\theta_{\ast}, \sigma^{2})(2e\theta_{\ast})^{k}  (k/e\sigma^{2})^{\frac{k}{2}}}{k^{k}}\\
    \leq& C(\theta_{\ast}, \sigma^{2})n^{-1} (2e^{\frac{1}{2}}\theta_{\ast}/\sigma)^{k} \leq C(\theta_{\ast}, \sigma^{2})n^{-\frac{3}{4}},
\end{align*}
since $(k+1)^{\frac{1}{2}}k^{k}e^{-k}\leq k!$ holds for $k \geq 0$ by Stirling's approximation. 

This is our first improvement regarding the convergence rate of $n$ since we could approximate $\ell_{\sigma}$ with explicit mixture density functions, compared to the approximation with general mixture density functions by the methods of functional analysis. This approximation error has the rate of $O(n^{-\frac{3}{4}})$ and able to be removed during the final computation.    \\

\par\noindent\textbf{Step 2.} Our next step is to analyze the bound of the second term of \eqref{eq: wasserstein} by computing the bounds of the distance between $h_{Q}$ and $h^{\text{obs}}$ and the distance between $h^{\text{obs}}$ and $h_{\hat{Q}_{d}}$. We use the property that $h^{\text{obs}}$ is based on the samples from $h_{Q}$ and $h_{\hat{Q}}$ minimizes the generalized distance with $h^{\text{obs}}$, which makes the approach plausible. We separate the given form 
\begin{align}\label{eq: distance}
\sum_{x = 0}^{k-1}b_{x}[h_{Q}(x)-h_{\hat{Q}_{d}}(x)]=\sum_{x = 0}^{k-1}b_{x}[h^{\text{obs}}(x) - h_{\hat{Q}_{d}}(x)]+\sum_{x = 0}^{k-1}b_{x}[h_{Q}(x) - h^{\text{obs}}(x)] 
\end{align}
and start with evaluating the bound of the first term of \eqref{eq: distance}. Our intuition is that since $\mathbb{E}[h^{\text{obs}}(x)] = h_{Q}(x)$ and $h_{Q}(x)$ and $h_{\hat{Q}_{d}}(x)$ are bounded by $\theta_{\ast}^{x}/ x!$, we could bound the coefficient $b_{x}$ with $h^{\text{obs}}(x)$, $h_{Q}(x)$, and $h_{\hat{Q}_{d}}(x)$. 

Regarding the bound of $\sum_{x=0}^{k-1}b_{x}[h^{\text{obs}}(x) - h_{\hat{Q}_{d}}(x)]$, we may factorize the term $h^{\text{obs}}(x) - h_{\hat{Q}_{d}}(x)$ and obtain   
\begin{align}
    \big \vert \sum_{x = 0}^{k-1}b_{x}[h^{\text{obs}}(x) - &h_{\hat{Q}_{d}}(x)]\big \vert\leq \sum_{x = 0}^{k-1} \big\vert\, b_{x}\,\big(\,\sqrt{h^{\text{obs}}(x)} + \sqrt{h_{\hat{Q}_{d}}(x)}\,\big)\,\big\vert \big\vert\, \sqrt{h^{\text{obs}}(x)} - \sqrt{h_{\hat{Q}_{d}}(x)} \,\big\vert\nonumber\\
    &\leq \big(\,\sum_{x = 0}^{k-1} b_{x}^{2}\,\big(\,\sqrt{h^{\text{obs}}(x)} + \sqrt{h_{\hat{Q}_{d}}(x)}\,\big)^{2}\,\big)^{\frac{1}{2}} \big(\,\sum_{x = 0}^{k-1}\big(\,\sqrt{h^{\text{obs}}(x)} - \sqrt{h_{\hat{Q}_{d}}(x)}\,\big)^{2}\,\big)^{\frac{1}{2}}\label{eq: hellinger}
\end{align}
by Cauchy-Schwarz inequality. We bound the first term of \eqref{eq: hellinger} by direct computation and bound the second term of \eqref{eq: hellinger} with Assumption \ref{assumption: chibound}. Since $h^{\text{obs}}(x)$ is nonnegative and $\mathbb{E}[h^{\text{obs}}(x)] = h_{Q}(x)$, $h^{\text{obs}}(x) \leq 4k h_{Q}(x)/\delta$
holds with at least probability $1-\delta/4k$ for all $0 \leq x \leq k-1$ by Markov's inequality. It then yields
\begin{align*}
    \sum_{x = 0}^{k-1} b_{x}^{2}\,\big(\,\sqrt{h^{\text{obs}}(x)} + \sqrt{h_{\hat{Q}_{d}}(x)}\,\big)^{2} \leq 2\sum_{x = 0}^{k-1} b_{x}^{2}(h^{\text{obs}}(x) + h_{\hat{Q}_{d}}(x))\leq \frac{8k}{\delta}\sum_{x = 0}^{k-1} b_{x}^{2}h_{Q}(x) +2\sum_{x = 0}^{k-1} b_{x}^{2} h_{\hat{Q}_{d}}(x)
\end{align*}
holds with at least probability $1-\delta/4$ for any $0<\delta<1$. Then, since $\vert b_{x} \vert\leq  C(\theta_{\ast}, \sigma^{2})(x+1)^{\frac{1}{2}} 2^{x} (1+ (x/e\sigma^{2})^{\frac{x}{2}})$ by Lemma \ref{lemma: polynomial}, $h_{Q}(x)\leq \theta_{\ast}^{x} /x!$, and $h_{\hat{Q}_{d}}(x)\leq \theta_{\ast}^{x} /x!$, we have
\begin{align*}
    \sum_{x = 0}^{k-1}b_{x}^{2}h_{Q}(x) \leq C(\theta_{\ast}, \sigma^{2}) \sum_{x = 0}^{k-1}\frac{(x+1) (4\theta_{\ast})^{x} (1+ (x/e\sigma^{2})^{\frac{x}{2}})^{2}}{x!}, 
\end{align*}
and similarly
\begin{align*}
    \sum_{x = 0}^{k-1}b_{x}^{2}h_{\hat{Q}_{d}}(x) \leq C(\theta_{\ast}, \sigma^{2}) \sum_{x = 0}^{k-1}\frac{(x+1) (4\theta_{\ast})^{x} (1+ (x/e\sigma^{2})^{\frac{x}{2}})^{2}}{x!}. 
\end{align*}
Combining all these inequalities, we obtain the bound 
\begin{align*}
    \sum_{x = 0}^{k-1} b_{x}^{2}\,\big(\,\sqrt{h^{\text{obs}}(x)} + \sqrt{h_{\hat{Q}_{d}}(x)}\,\big)^{2}\leq\frac{C(\theta_{\ast},\sigma^{2})k}{\delta} \sum_{x = 0}^{k-1}\frac{(x+1) (4\theta_{\ast})^{x} (1+ (x/e\sigma^{2})^{\frac{x}{2}})^{2}}{x!}
\end{align*}
holds with at least probability $1-\delta/4$ for any $0<\delta<1$. We could bound
\begin{align*}
   \sum_{x = 0}^{k-1}\frac{(x+1) (4\theta_{\ast})^{x} (1+ (x/e\sigma^{2})^{\frac{x}{2}})^{2}}{x!} \leq k\,\big(\, \sum_{x = 0}^{k-1}\frac{(4\theta_{\ast})^{x}}{x!}+\sum_{x=0}^{k-1}(4\theta_{\ast}/\sigma^{2})^{x} \,\big)\leq C(\theta_{\ast}, \sigma^{2}) k(1+(4\theta_{\ast}/\sigma^{2})^{k})
\end{align*}
by Stirling's approximation $(k+1)^{\frac{1}{2}}k^{k}e^{-k}\leq k!$, and it leads to the bound of the first term of \eqref{eq: hellinger}
\begin{align*}
    \sum_{x = 0}^{k-1} b_{x}^{2}\,\big(\,\sqrt{h^{\text{obs}}(x)} + \sqrt{h_{\hat{Q}_{d}}(x)}\,\big)^{2}&\leq C(\theta_{\ast},\sigma^{2})\delta^{-1}k^{2}(1+ (4\theta_{\ast}/\sigma^{2})^{k})\\
    &\leq C(\theta_{\ast},\sigma^{2})\delta^{-1}(\log n)^{2}(1+ n^{\frac{16 \log 4\theta_{\ast}\sigma^{-2}}{\log \log n}}). 
\end{align*}
Here, the upper bound has smaller asymptotics than any polynomial rate of $n$ and this is a second improvement of this paper. 

 Now, we derive the upper bound of the second term of \eqref{eq: hellinger} with the Assumption \ref{assumption: chibound} and the Lemma \ref{lemma: chi-square}. We have the relation between squared Hellinger distance, generalized distance $d$, and $\chi^{2}$-divergence, so we could effectively derive the bound of the squared Hellinger distance between the observed distribution and minimum-distance estimator based on the bound of $\chi^{2}$-divergence. By the Assumption \ref{assumption: chibound}, 
\begin{align*}
    \sum_{x = 0}^{k-1}\big(\,\sqrt{h^{\text{obs}}(x)} - \sqrt{h_{\hat{Q}_{d}}(x)}\,\big)^{2}& \leq 2H^{2}(h^{\text{obs}}, h_{\hat{Q}_{d}})\leq 2\gamma_{3}^{-1}d(h^{\text{obs}}\Vert h_{\hat{Q}_{d}})\leq 2\gamma_{3}^{-1}d(h^{\text{obs}}\Vert h_{Q})
\end{align*}
holds, and the last inequality is from the minimimum distance property of the estimator $\hat{Q}_{d}$. Regarding the bound of $d(h^{\text{obs}}\Vert h_{Q})$, by Lemma \ref{lemma: chi-square},  
\begin{align*}
        d(h^{\text{obs}}\Vert h_{Q})\leq &\gamma_{4}\chi^{2}(h^{\text{obs}}\Vert h_{Q})\leq \gamma_{4}Cn^{-1}(\log n)^{\theta_{\ast}+4}\delta^{-1-\epsilon}
\end{align*}
with probability at least $1-\delta/4$ for any $0<\delta<1$ and $\epsilon = \log \log n / \log n$ with sufficient constant $C$. It leads to the bound of the first term of \eqref{eq: distance}  
\begin{align*}
    \big \vert \sum_{x = 0}^{k-1}b_{x}[h^{\text{obs}}(x) - h_{\hat{Q}_{d}}(x)]\big \vert
    &\leq \big(\,\sum_{x = 0}^{k-1} b_{x}^{2}\,\big(\,\sqrt{h^{\text{obs}}(x)} + \sqrt{h_{\hat{Q}_{d}}(x)}\,\big)^{2}\,\big)^{\frac{1}{2}} \big(\,\sum_{x = 0}^{k-1}\big(\,\sqrt{h^{\text{obs}}(x)} - \sqrt{h_{\hat{Q}_{d}}(x)}\,\big)^{2}\,\big)^{\frac{1}{2}}\\
    &\leq n^{-\frac{1}{2}}(\log n)^{\frac{\theta_{\ast}}{2}+3}(1+ n^{\frac{8 \log 4\theta_{\ast}\sigma^{-2}}{\log \log n}})C(\theta_{\ast}, \sigma^{2}, \gamma_{3}, \gamma_{4})\delta^{-1-\frac{\epsilon}{2}}
\end{align*}
with probability at least $1-\delta/2$ for any $0<\delta<1$. 

We apply the similar derivation regarding the bound of the second term of \eqref{eq: distance}. We obtain the bound by the factorization of $h_{Q}(x) - h^{\text{obs}}(x)$ and the Cauchy-Schwarz inequality, 
\begin{align}
    \big \vert \sum_{x = 0}^{k-1}b_{x}[h^{\text{obs}}(x) - &h_{Q}(x)]\big \vert\leq \sum_{x = 0}^{k-1} \big\vert\, b_{x}\,\big(\,\sqrt{h^{\text{obs}}(x)} + \sqrt{h_{Q}(x)}\,\big)\,\big\vert \big\vert\, \sqrt{h^{\text{obs}}(x)} - \sqrt{h_{Q}(x)} \,\big\vert\nonumber\\
    &\leq \big(\,\sum_{x = 0}^{k-1} b_{x}^{2}\,\big(\,\sqrt{h^{\text{obs}}(x)} + \sqrt{h_{Q}(x)}\,\big)^{2}\,\big)^{\frac{1}{2}} \big(\,\sum_{x = 0}^{k-1}\big(\,\sqrt{h^{\text{obs}}(x)} - \sqrt{h_{Q}(x)}\,\big)^{2}\,\big)^{\frac{1}{2}}\label{eq: hellinger_2}
\end{align}
where the first term of \eqref{eq: hellinger_2} is similarly bounded by 
\begin{align*}
    \sum_{x = 0}^{k-1} b_{x}^{2}\,\big(\,\sqrt{h^{\text{obs}}(x)} + \sqrt{h_{Q}(x)}\,\big)^{2} \leq 2\sum_{x = 0}^{k-1} b_{x}^{2}(h^{\text{obs}}(x) + h_{Q}(x))
    \leq C(\theta_{\ast},\sigma^{2})\delta^{-1}(\log n)^{2}(1+ n^{\frac{16 \log 4\theta_{\ast}\sigma^{-2}}{\log \log n}}) 
\end{align*}
with probability at least $1-\delta/4$ for any $0<\delta<1$. In addition, by the information theoretical inequality and Lemma \ref{lemma: chi-square}, the bound of the second term of \eqref{eq: hellinger_2} 
\begin{align*}
    \sum_{x = 0}^{k-1}(\sqrt{h^{\text{obs}}(x)} - \sqrt{h_{Q}(x)})^{2}& \leq 2H^{2}(h^{\text{obs}}, h_{Q})\leq \chi^{2}(h^{\text{obs}}\Vert h_{Q})\leq Cn^{-1}(\log n)^{\theta_{\ast}+4}\delta^{-1-\epsilon}
\end{align*}
holds, and it gives the bound of the second term of \eqref{eq: distance}
\begin{align*}
    \big \vert \sum_{x = 0}^{k-1}b_{x}[h^{\text{obs}}(x) - h_{\hat{Q}_{d}}(x)]\big \vert
    \leq n^{-\frac{1}{2}}(\log n)^{\frac{\theta_{\ast}}{2}+3}(1+ n^{\frac{8 \log 4\theta_{\ast}\sigma^{-2}}{\log \log n}})C(\theta_{\ast}, \sigma^{2})\delta^{-1-\frac{\epsilon}{2}}
\end{align*}
with probability at least $1-\delta / 2$.\\

\par\noindent\textbf{Step 3.}
In this step, we obtain the overall bound of $\mathbb{E}W_{1}^{\sigma}(Q, \hat{Q}_{d})$ based on the results of Step 1 and Step 2. Combining all the inequalities, we have the probabilistic upper bound condition of $W_{1}^{\sigma}(Q, \hat{Q}_{d})$ that
\begin{align*}
    W_{1}^{\sigma}(Q, \hat{Q}_{d})\leq C(\theta_{\ast}, \sigma^{2}, \gamma_{3}, \gamma_{4}) n^{-\frac{1}{2}}\delta^{-1-\frac{\epsilon}{2}}(n^{-\frac{1}{4}}+(\log n)^{\frac{\theta_{\ast}}{2}+3}(1+ n^{\frac{8 \log 4\theta_{\ast}\sigma^{-2}}{\log \log n}}))
\end{align*}
holds with probability at least $1-\delta$ for any $0<\delta<1$. It leads that for all $n\geq N$ and $0<\delta<1$,
\begin{align*}
    W_{1}^{\sigma}(Q, \hat{Q}_{d})\leq C(\theta_{\ast}, \sigma^{2}, \gamma_{3}, \gamma_{4})n^{-\frac{1}{2}}P(n)\delta^{-1-\frac{\epsilon}{2}}
\end{align*}
holds with probability at least $1-\delta$, where $P(n)$ is given by
\begin{align*}
    P(n) = (\log n)^{\frac{\theta_{\ast}}{2}+3}(1+ n^{\frac{8 \log 4\theta_{\ast}\sigma^{-2}}{\log \log n}})
\end{align*}
which is smaller than the polynomial rate. We denote $C(\theta_{\ast}, \sigma^{2}, \gamma_{3}, \gamma_{4})$ as $C$ during the rest of the proof for convenience. By \citet[Section 5]{Goldfeld24}, $W_{1}^{\sigma}(Q, \hat{Q}_{d})$ is nonnegative and upper bounded by 
\begin{align*}
    W_{1}^{\sigma}(Q, \hat{Q}_{d})\leq W_{1}(Q, \hat{Q}_{d})\leq \theta_{\ast},
\end{align*}
so by the property of the expectation we obtain 
\begin{align*}
    \mathbb{E}W_{1}^{\sigma}(Q, \hat{Q}_{d})&=\int_{0}^{\theta_{\ast}}\mathbb{P}(W_{1}^{\sigma}(Q, \hat{Q}_{d})>x)\d x\\
    &\leq\int_{0}^{ C2^{1+\frac{\epsilon}{2}}n^{-\frac{1}{2}}P(n)} \,1\,\d x+\int_{ C2^{1+\frac{\epsilon}{2}}n^{-\frac{1}{2}}P(n)}^{\theta_{\ast}} \mathbb{P}(W_{1}^{\sigma}(Q, \hat{Q}_{d})>x)\d x\\
    &\leq C2^{1+\frac{\epsilon}{2}}n^{-\frac{1}{2}}P(n)+ \int_{ C2^{1+\frac{\epsilon}{2}}n^{-\frac{1}{2}}P(n)}^{\theta_{\ast}} \mathbb{P}(W_{1}^{\sigma}(Q, \hat{Q}_{d})>x)\d x,
\end{align*}
and by direct computation with the upper bound condition of $W_{1}^{\sigma}(Q, \hat{Q}_{d})$,
\begin{align*}
    &\int_{ C2^{1+\frac{\epsilon}{2}}n^{-\frac{1}{2}}P(n)}^{\theta_{\ast}}\mathbb{P}(W_{1}^{\sigma}(Q, \hat{Q}_{d})>x)\d x\\
    \leq& \int_{C2^{1+\frac{\epsilon}{2}}n^{-\frac{1}{2}}P(n)}^{\theta_{\ast}} \mathbb{P}(W_{1}^{\sigma}(Q, \hat{Q}_{d})>C\delta^{-1-\frac{\epsilon}{2}}n^{-\frac{1}{2}}P(n))\d (C\delta^{-1-\frac{\epsilon}{2}}n^{-\frac{1}{2}}P(n))\\
    \leq& \int_{\frac{1}{2}}^{(C\theta_{\ast}^{-1}n^{-\frac{1}{2}}P(n))^{\frac{2}{2+\epsilon}}} \delta C(-1-\frac{\epsilon}{2})\delta^{-2-\frac{\epsilon}{2}}n^{-\frac{1}{2}}P(n)\d \delta\\
    =& C \frac{2+\epsilon}{\epsilon}n^{-\frac{1}{2}}P(n)\{(C\theta_{\ast}^{-1}n^{-\frac{1}{2}}P(n))^{-\frac{\epsilon}{2+\epsilon}} - 2^{\frac{\epsilon}{2}}\}\\
    \leq& C n^{-\frac{1}{2}}n^{\frac{\epsilon}{2(2+\epsilon)}}\log n P(n)^{\frac{2}{2+\epsilon}} \leq C n^{-\frac{1}{2}}(\log n)^{\frac{5}{4}} P(n)
\end{align*}
holds. Therefore, the final upper bound of $\mathbb{E}W_{1}^{\sigma}(Q, \hat{Q}_{d})$ is given by
\begin{align*}
    \mathbb{E}W_{1}^{\sigma}(Q, \hat{Q}_{d}) \leq C n^{-\frac{1}{2}}(\log n)^{\frac{\theta_{\ast}}{2}+5}(1+ n^{\frac{8 \log 4\theta_{\ast}\sigma^{-2}}{\log \log n}})
\end{align*}
and since 
\begin{align*}
    1+ n^{\frac{8 \log 4\theta_{\ast}\sigma^{-2}}{\log \log n}} = 
    \begin{cases}
        O(n^{\frac{8 \log 4\theta_{\ast}\sigma^{-2}}{\log \log n}}),\,\,\,\,\,&\text{if}\,\,\sigma^{2} < 4\theta_{\ast},\\
        O(1), \,\,\,\,\,&\text{if}\,\,\sigma^{2}\geq 4 \theta_{\ast},
    \end{cases}
\end{align*}
we obtain the upper bound of $\mathbb{E}W_{1}^{\sigma}(Q, \hat{Q}_{d})$ as
\begin{align*}
    \sup_{Q \in \mathcal{P}([0, \theta_{\ast}])}\mathbb{E}W_{1}^{\sigma}(Q, \hat{Q}_{d})\leq
    \begin{cases}
       C(\theta_{\ast}, \sigma^{2}, \gamma_{3}, \gamma_{4})\cdot n^{-\frac{1}{2}+\frac{8 \log 4\theta_{\ast}\sigma^{-2}}{\log \log n}}(\log n)^{\frac{\theta_{\ast}}{2}+5},\,\,\,\,\,&\text{if}\,\,\sigma^{2}<4\theta_{\ast},\\
        C(\theta_{\ast}, \sigma^{2}, \gamma_{3}, \gamma_{4})\cdot n^{-\frac{1}{2}}(\log n)^{\frac{\theta_{\ast}}{2}+5}, &\text{if}\,\,\sigma^{2} \geq 4\theta_{\ast},
    \end{cases}
\end{align*}
which completes the proof. 
\end{proof}

\begin{proof}[Proof of Proposition \ref{proposition: distance}:]
We start with the proof that Assumption \ref{assumption: structure} holds for given generalized distances. Based on \citet[Section 2.1]{Jana22} and direct computation, we have a set of functions $w(p)$ and $\phi(p(x),q(x))$ for the following generalized distances:
\begin{itemize}
    \item squared Hellinger distance $H^{2}(p,q)$: $w(p) = 1$ and $\phi(p(x), q(x)) = -\sqrt{p(x)q(x)}$;
    \item Le Cam distance $\text{LC}(p, q)$: $w(p) = 1$ and $\phi(p(x), q(x)) = -2p(x)q(x)/(p(x)+q(x))$;
    \item Jensen-Shannon divergence $\text{JS}(p, q)$: $w(p) = 2 \log 2$ and $\phi(p(x), q(x)) = p(x) \log (p(x)/(p(x)+q(x))+ q(x) \log (q(x)/(p(x)+q(x))$;
    \item KL-divergence $\text{KL}(p\Vert q)$: $w(p) = 0$ and $\phi(p(x), q(x)) = p(x) \log (p(x)/q(x))$; 
    \item chi-square divergence $\chi^{2}(p\Vert q)$: $w(p) = -1$ and $\phi(p(x), q(x)) = p(x)^{2}/q(x)$,
\end{itemize}
with the form of \eqref{eq:distance form} that satisfies the conditions of Assumption \ref{assumption: structure}.

Regarding Assumptions \ref{assumption: KLbound} and \ref{assumption: chibound}, we first note the following relation between the squared Hellinger distance, KL-divergence, and chi-square divergence in \citet[Section 7.6]{Polyanskiy23}:
\begin{align}\label{eq: distance bound}
    2H^{2}(p, q) \leq \text{KL}(p \Vert q) \leq \chi^{2}(p\Vert q).
\end{align}
It yields that the squared Hellinger distance and KL-divergence both satisfy Assumptions \ref{assumption: KLbound} and \ref{assumption: chibound}, and the chi-square divergence satisfies Assumption \ref{assumption: chibound}. Regarding Le Cam distance and Jensen-Shannon divergence, we also have two inequalities from \citet[Section 7.6]{Polyanskiy23} that 
\begin{align*}
    H^{2}(p, q)\leq \text{LC}(p, q)\leq 2H^{2}(p, q)  
\end{align*}
and
\begin{align*}
    \text{LC}(p, q) \leq \text{JS}(p, q) \leq 2 \log 2 \cdot \text{LC}(p, q).  
\end{align*}
Combining these inequalities with \eqref{eq: distance bound} implies that Le Cam distance and Jensen-Shannon divergence also satisfy Assumptions \ref{assumption: KLbound} and \ref{assumption: chibound}. 
\end{proof}

\subsection{Auxiliary lemmas}

\begin{lemma}\label{lemma: polynomial}

Recall the $1$-Lipschitz function $\ell$ with $\ell(0) = 0$ and convoluted function $\ell_{\sigma}= \ell \ast \varphi_{\sigma}$. Then, there exists polynomial $P_{k-1}(\theta) = \sum_{x = 0}^{k-1} c_{x} \theta^{x}$ that has the bounded coefficients 
\begin{align*}
    \vert c_{x}\vert \leq \frac{C(\theta_{\ast}, \sigma^{2})(x+1)^{\frac{1}{2}} 2^{x} (1+ (x/e\sigma^{2})^{\frac{x}{2}})}{x!} 
\end{align*}
and $e^{-\theta}P_{k-1}(\theta)$ approximates the function $\ell_{\sigma}(\theta)-\ell_{\sigma}(0)$
\begin{align*}
    \sup_{\theta \in [0, \theta_{\ast}]} \vert \ell_{\sigma}(\theta) - \ell_{\sigma}(0) - e^{-\theta} P_{k-1}(\theta)\vert 
    \leq \frac{C(\theta_{\ast}, \sigma^{2})(k+1)^{\frac{1}{2}} (2\theta_{\ast})^{ k} (1+ (k/e\sigma^{2})^{\frac{k}{2}})}{k!}
\end{align*}
with sufficient constant $C(\theta_{\ast}, \sigma^{2})>0$.
\end{lemma}
\begin{proof} In this proof, we consider the Taylor polynomial for the function $v(\theta) = e^{\theta}(\ell_{\sigma}(\theta) - \ell_{\sigma}(0))$ and bound the coefficients by bounding the derivatives of $v$.\\

\par\noindent\textbf{Step 1.} In this step, we bound the derivative of $v$ with the properties of the Hermite Polynomial. Denote the $x$-th order Hermite polynomial as $He_{x}(\theta)$ for $x \geq 0$, where we refer the notions from \cite[Section 22]{Abramowitz64}. By the property of Hermite polynomials, derivative $\ell_{\sigma}^{(x)}(\theta)$ is computed as
\begin{align*}
    \vert \ell_{\sigma}^{(x)}(\theta) \vert &= \Big\vert \,\int_{\mathbb{R}} \ell(u)\,\varphi_{\sigma}^{(x)}(\theta-u)\, \d u\, \Big \vert=\sigma^{-x}\,\Big\vert \,\int_{\mathbb{R}} \ell(u)\varphi_{\sigma}(\theta-u) H_{x}((\theta-u)/\sigma) \d u\, \Big \vert\\
    & = \sigma^{-x}\,\Big\vert \,\int_{\mathbb{R}} \ell(\theta - u)\varphi_{\sigma}(u) H_{x}(u/\sigma) \d u\, \Big \vert
\end{align*}
which is the integral of the Hermite polynomial. Then, by the Cauchy-Schwarz inequality,
\begin{align*}
    \Big\vert \,\int_{\mathbb{R}} \ell(\theta - u)\varphi_{\sigma}(u) H_{x}(u/\sigma) \d u\, \Big \vert&\leq \big(\,\int_{\mathbb{R}} \ell(\theta - u)^{2} \varphi_{\sigma}(u) \d u \,\big )^{\frac{1}{2}} \big(\,\int_{\mathbb{R}} H_{x}(u/\sigma)^{2} \varphi_{\sigma}(u) \d u\,\big)^{\frac{1}{2}}\\
    &\leq \sqrt{x!}\,\big(\,\int_{\mathbb{R}} \vert \theta - u\vert^{2}\varphi_{\sigma}(u) \d u\,\big)^{\frac{1}{2}}= \sqrt{x!}\,\sqrt{\theta^{2}+\sigma^{2}} 
\end{align*}
holds, and it yields $\vert \ell_{\sigma}^{(x)}(\theta) \vert \leq \sigma^{-x} \sqrt{x!}\,\sqrt{\theta^{2}+\sigma^{2}}$. We refer these computation of Hermite polynomial to \citet[Section 5]{Han23}. 

Then, we compute the bound of the derivative of $v^{(k)}(\theta)$ with the bound of the derivatives of $\ell_{\sigma}^{(x)}(\theta)$. For any $\theta \in [0, \theta_{\ast}]$, by direct computation,
\begin{align*}
\vert v^{(k)}(\theta)\vert &= \big\vert \,\sum_{x = 0}^{k}\binom{k}{x}\,e^{\theta} \ell_{\sigma}^{(x)}(\theta) - e^{\theta}  \ell_{\sigma}(0)\,\big\vert \\
&\leq \sum_{x = 0}^{k}\binom{k}{x}\,e^{\theta}\, \vert \ell_{\sigma}^{(x)}(\theta)\vert + e^{\theta} \,\vert \ell_{\sigma}(0)\vert\leq \sum_{x = 0}^{k}\binom{k}{x} \,e^{\theta_{\ast}} \sigma^{-x}\,\sqrt{x!}\,\sqrt{\theta_{\ast}^{2}+\sigma^{2}} + e^{\theta_{\ast}} \sigma\\
&\leq e^{\theta_{\ast}+1}\,\sqrt{\theta_{\ast}^{2}+\sigma^{2}}\,(k+1)^{\frac{1}{2}} \sum_{x = 0}^{k}\binom{k}{x} \,(k/e\sigma^{2})^{\frac{x}{2}}+ C(\theta_{\ast}, \sigma^{2})\\
&\leq C(\theta_{\ast}, \sigma^{2})(k+1)^{\frac{1}{2}} (1+(k/e\sigma^{2})^{\frac{1}{2}})^{k}\leq C(\theta_{\ast}, \sigma^{2})(k+1)^{\frac{1}{2}} 2^{k-1} (1+ (k/e\sigma^{2})^{\frac{k}{2}})
\end{align*}
since $x!<x^{x}(x+1)e^{-x+2}$ for $x \geq 0$ by Stirling's approximation. Therefore, we have the bound of $\vert v^{(k)}(\theta)\vert$ as a function of $k$. \\

\par\noindent\textbf{Step 2.} We now construct the polynomial $P_{k}(\theta) = \sum_{x = 0}^{k}c_{x}\theta^{x}$ with $c_{x} = v^{(x)}(0)/x!$ as a Taylor polynomial of $v(\theta)$ and prove that it satisfies the conditions in the statement. Note that coefficients $c_{x}$ are bounded as
\begin{align*}
    \vert c_{x}\vert = \frac{\vert v^{(x)}(0)\vert}{x!} \leq \frac{C(\theta_{\ast}, \sigma^{2})(x+1)^{\frac{1}{2}} 2^{x-1} (1+ (x/e\sigma^{2})^{\frac{x}{2}})}{x!} 
\end{align*}
based on the derivation of Step 1. In addition, by Taylor's theorem, for any $\theta \in [0, \theta_{\ast}]$, there exists $\xi(\theta)\in[0, \theta_{\ast}]$ such that the error term is bounded as
\begin{align*}
    \vert v(\theta) - P_{k-1}(\theta) \vert \leq \frac{\vert v^{(k)}(\xi(\theta)) \vert}{k!}\theta^{k} \leq \frac{C(\theta_{\ast}, \sigma^{2})(k+1)^{\frac{1}{2}} 2^{k-1} (1+ (k/e\sigma^{2})^{\frac{k}{2}})}{k!}\theta_{\ast}^{k},
\end{align*}
which yields
\begin{align*}
    \sup_{\theta\in[0, \theta_{\ast}]}\vert v(\theta) - P_{k-1}(\theta) \vert\leq \frac{C(\theta_{\ast}, \sigma^{2})(k+1)^{\frac{1}{2}} 2^{k-1} (1+ (k/e\sigma^{2})^{\frac{k}{2}})}{k!}\theta_{\ast}^{k}.
\end{align*}
Finally, since $e^{-\theta}$ is bounded by constants,
\begin{align*}
    \sup_{\theta \in [0, \theta_{\ast}]} \vert \ell_{\sigma}(\theta) - \ell_{\sigma}(0) - e^{-\theta} P_{k-1}(\theta)\vert 
    \leq &\sup_{\theta \in [0, \theta_{\ast}]}e^{-\theta} \vert v(\theta) - P_{k-1}(\theta)\vert\leq \sup_{\theta \in [0, \theta_{\ast}]}\vert v(\theta) - P_{k-1}(\theta)\vert,
\end{align*}
so the polynomial $P_{k-1}(\theta)$ satisfies the condition of the statement.  
\end{proof}

\begin{lemma}\label{lemma: chi-square}

We consider the Poisson mixture model and establish the bound for $\chi^{2}(h^{\text{obs}}\Vert h_{Q})$. For any $\delta \in (0,1)$ and $\epsilon(n) = \log \log n / \log n$, $\chi^{2}(h^{\text{obs}}\Vert h_{Q}) \leq C  n^{-1}(\log n)^{\theta_{\ast}+4} \delta^{-1-\epsilon}$ holds with probability $1-\delta$ with sufficient constant $C>0$.   
\end{lemma}
\begin{proof}
The proof is based on the approach of \citet[Lemma 17]{Han23}, which includes the bound of the partial sum on $\chi^{2}(h^{\text{obs}}\Vert h_{Q})$. We first bound $\chi^{2}(h^{\text{obs}}\Vert h_{Q})$ with H\"older's inequality then derive the explicit bound as a function of $n$. We fix the constants as $\epsilon = \epsilon(n) = \log \log n / \log n$, $A = A(n) = (\log n)^{1/3}>1$, and $\delta \in(0,1)$. \\

\par\noindent\textbf{Step 1.} We start with the bound of $\chi^{2}(h^{\text{obs}}\Vert h_{Q})$ with the H\"older's inequality. The $\chi^{2}$-distance is bounded as 
\begin{align}\label{eq: Holder}
    \sum_{x\geq 0}\frac{(h^{\text{obs}}(x) - h_{Q}(x))^{2}}{h_{Q}(x)}\leq \big(\,\sum_{x\geq 0}\frac{(h^{\text{obs}}(x) - h_{Q}(x))^{2}}{h_{Q}(x)}A^{-\frac{\epsilon x}{1-\epsilon}}\,\big)^{1-\epsilon}\big(\,\sum_{x\geq 0}\frac{(h^{\text{obs}}(x) - h_{Q}(x))^{2}}{h_{Q}(x)}A^{x}\,\big)^{\epsilon}
\end{align}
by the H\"older's inequality, and we bound the first and second term of the right side. 

Regarding the first term of \eqref{eq: Holder}, we have the bound of the expectation
\begin{align*}
    \mathbb{E}\big\{\sum_{x\geq 0}\frac{(h^{\text{obs}}(x) - h_{Q}(x))^{2}}{h_{Q}(x)}A^{-\frac{\epsilon x}{1-\epsilon}}\big\} = \sum_{x\geq 0}A^{-\frac{\epsilon x}{1-\epsilon}} \mathbb{E}\big\{\frac{(h^{\text{obs}}(x) - h_{Q}(x))^{2}}{h_{Q}(x)}\big\}
    \leq \frac{1}{n(1 - A^{-\frac{\epsilon}{1-\epsilon}})}. 
\end{align*}
It leads to the bound of the first term by the Markov's inequality that
\begin{align*}
    \sum_{x\geq 0}\frac{(h^{\text{obs}}(x) - h_{Q}(x))^{2}}{h_{Q}(x)}A^{-\frac{\epsilon x}{1-\epsilon}}\leq \frac{1}{n\delta(1 - A^{-\frac{\epsilon}{1-\epsilon}})}
\end{align*}
holds with probability at least $1-\delta$. Regarding the second term of \eqref{eq: Holder}, we have the bound by the direct computation
\begin{align}\label{eq: Holder_second term}
    \sum_{x\geq 0}\frac{(h^{\text{obs}}(x) - h_{Q}(x))^{2}}{h_{Q}(x)}A^{x}\leq 2 \sum_{x\geq 0}\frac{h^{\text{obs}}(x)^{2}}{h_{Q}(x)}A^{x} +2\sum_{x\geq 0}h_{Q}(x)A^{x},    
\end{align}
and the bound of the second term of \eqref{eq: Holder_second term}
\begin{align*}
    \sum_{x\geq 0}h_{Q}(x)A^{x} = \sum_{x\geq 0}\int_{0}^{\theta_{\ast}} f(x; \theta)\d Q(\theta) A^{x}\leq \sum_{x\geq 0} \frac{\theta_{\ast}^{x}}{x!}A^{x}\leq e^{A\theta_{\ast}} 
\end{align*}
is obtained. Regarding the first term of \eqref{eq: Holder_second term}, again by Markov's inequality,
\begin{align*}
&\mathbb{P}(\text{there exists some } x\geq 0 \text{ such that }h^{\text{obs}}(x)>kh_{Q}(x)A^{x}\,)\\ 
\leq& \sum_{x\geq 0} \mathbb{P}(\,h^{\text{obs}}(x)>kh_{Q}(x)A^{x})\leq \frac{1}{k}\sum_{x\geq 0}A^{-x} = \frac{1}{k(1-A^{-1})}
\end{align*}
holds for $k = 1/(\delta(1-A^{-1}))>0$, so we have the bound of the first term of \eqref{eq: Holder_second term} as
\begin{align*}
    \sum_{x\geq 0}\frac{h^{\text{obs}}(x)^{2}}{h_{Q}(x)}A^{x}\leq k^{2}\sum_{x\geq 0}h_{Q}(x)A^{3x} \leq\frac{e^{A^{3}\theta_{\ast}}}{\delta^{2}(1-A^{-1})^{2}}
\end{align*}
with probability at least $1-\delta$. Therefore, combining the bounds of the two terms, we obtain the final bound of $\chi^{2}(h^{\text{obs}}\Vert h_{Q})$ as
\begin{align}
    \sum_{x\geq 0}\frac{(h^{\text{obs}}(x) - h_{Q}(x))^{2}}{h_{Q}(x)}&\leq\frac{1}{n^{1-\epsilon}\delta^{1-\epsilon}} \frac{2^{\epsilon}}{(1 - A^{-\frac{\epsilon}{1-\epsilon}})^{1-\epsilon}}(e^{A\theta_{\ast}}+\frac{e^{A^{3}\theta_{\ast}}}{\delta^{2}(1-A^{-1})^{2}})^{\epsilon}\nonumber\\
    &\leq\frac{1}{n^{1-\epsilon}\delta^{1+\epsilon}} \frac{2^{2\epsilon}}{(1 - A^{-\frac{\epsilon}{1-\epsilon}})^{1-\epsilon}}\frac{e^{A^{3}\theta_{\ast}\epsilon}}{(1-A^{-1})^{2\epsilon}}\label{eq: chi-square}   
\end{align}
with probability at least $1-\delta$.\\

\par\noindent\textbf{Step 2.} In this step, we recall $\epsilon = \log \log n / \log n$ and $A = (\log n)^{1/3}>1$ and provide more clear form of the upper bound \eqref{eq: chi-square} with respect to $n$ by direct computation. First, regarding the numerator of the upper bound \eqref{eq: chi-square}, $2^{2\epsilon}\leq 4$ and 
\begin{align*}
    e^{A^{3}\theta_{\ast}\epsilon}= e^{\theta_{\ast}\log \log n} = (\log n)^{\theta_{\ast}}
\end{align*}
holds. Then, considering the denominator of \eqref{eq: chi-square},     
\begin{align*}
    \log (\log n)^{\frac{\epsilon}{3}}= \frac{\epsilon}{3}\log \log n \geq \frac{1}{\log n - 1}\geq \log (1+\frac{1}{\log n - 1})\geq \log (\frac{\log n}{\log n - 1})
\end{align*}
holds for $\log \log n \geq \sqrt{6}$, so
\begin{align*}
    A^{-\frac{\epsilon}{1-\epsilon}}\leq A^{-\epsilon} = (\log n)^{-\frac{\epsilon}{3}}\leq 1 - \frac{1}{\log n}\leq 1 - \frac{1}{(\log n)^{\frac{1}{1-\epsilon}}} 
\end{align*}
and it leads to the bound of the first term of the denominator $(1 - A^{-\frac{\epsilon}{1-\epsilon}})^{-1+\epsilon} \leq \log n$. Regarding the second term of the denominator,
\begin{align*}
    \log \frac{1}{(1-A^{-1})^{2\epsilon}}=2\epsilon \log(1+\frac{A^{-1}}{1-A^{-1}})\leq \frac{2A^{-1}\epsilon}{1-A^{-1}}\leq \frac{2\epsilon}{A-1}\leq 2\log \log n
\end{align*}
holds for $\log n \geq 8$, which derives $(1-A^{-1})^{-2\epsilon}\leq (\log n)^{2}$ for $\log n \geq 8$. It implies that 
\begin{align*}
    \frac{2^{2\epsilon}}{(1 - A^{-\frac{\epsilon}{1-\epsilon}})^{1-\epsilon}}\frac{e^{A^{3}\theta_{\ast}\epsilon}}{(1-A^{-1})^{2\epsilon}}\leq 4 (\log n)^{\theta_{\ast}+3} 
\end{align*}
holds for all sufficiently large $n$. Finally, since $n^{-1+\epsilon} = n^{-1} \log n$, we complete the proof. 
\end{proof}

\bibliographystyle{apalike}
\bibliography{smooth}

\end{document}